\documentclass[11pt]{amsart}

\usepackage{amsmath,amssymb,amsfonts,mathtools}
\usepackage{amsthm}
\usepackage{enumitem}
\usepackage[colorlinks=true,citecolor=blue,linkcolor=blue,urlcolor=blue]{hyperref}

\numberwithin{equation}{section}

\newtheorem{theorem}{Theorem}[section]
\newtheorem{lemma}[theorem]{Lemma}
\newtheorem{proposition}[theorem]{Proposition}
\newtheorem{corollary}[theorem]{Corollary}

\theoremstyle{definition}
\newtheorem{definition}[theorem]{Definition}
\newtheorem{example}[theorem]{Example}

\theoremstyle{remark}
\newtheorem{remark}[theorem]{Remark}

\usepackage[margin=1.5in]{geometry}

\ifx
\usepackage{amssymb,mathrsfs,graphicx}
\usepackage[colorlinks=true, pdfstartview=FitV, linkcolor=blue,citecolor=deepred, urlcolor=RoyalBlue]{hyperref}
\usepackage[margin=1in]{geometry}
\usepackage{sidecap}
\usepackage{rotating,dsfont}
\usepackage{enumitem}
\usepackage[normalem]{ulem}

\usepackage{colortbl}
\definecolor{black}{rgb}{0.0, 0.0, 0.0}
\definecolor{red}{rgb}{1.0, 0.5, 0.5}
\definecolor{deepred}{rgb}{0.7, 0, 0} 

\provideboolean{shownotes} 
\setboolean{shownotes}{true} 
\newcommand{\margnote}[1]{
\ifthenelse{\boolean{shownotes}}%
{\marginpar{\raggedright\tiny\texttt{#1}}}%
{}%
}
\newcommand{\hole}[1]{
\ifthenelse{\boolean{shownotes}}%
{\begin{center} \fbox{ \rule {.25cm}{0cm} \rule[-.1cm]{0cm}{.4cm}
\parbox{.85\textwidth}{\begin{center} \texttt{#1}\end{center}} \rule
{.25cm}{0cm}}\end{center}} {} }
\fi

\title[Finite-time breakdown of the Euler--alignment system]{Finite-time breakdown \\ of the Euler--alignment system \\ for supercritical initial data}

\thanks{The work of YPC was supported by NRF grant no. 2022R1A2C1002820 and RS-2024-00406821.\newline
\hspace*{0.3cm} The work of ET was supported by ONR grant  N00014-2412659 and NSF grant DMS-2508407.} 

\author[Young-Pil Choi]{Young-Pil Choi}
\address{Department of Mathematics, Yonsei University\newline 
\hspace*{0.3cm} Seoul 03722, Republic of Korea}
\email{ypchoi@yonsei.ac.kr}

\author[Eitan Tadmor]{Eitan Tadmor}
\address{Department of Mathematics and IPST, University of Maryland\newline 
\hspace*{0.3cm} College Park, MD 20742, USA}
\email{tadmor@umd.edu}

\ifx
\numberwithin{equation}{section}
\newtheorem{theorem}{Theorem}[section]
\newtheorem{lemma}[theorem]{Lemma}
\newtheorem{proposition}[theorem]{Proposition}
\newtheorem{corollary}[theorem]{Corollary}
\newtheorem{remark}[theorem]{Remark}

\fi

\newcommand{\R}{\mathbb R}

\newcommand{\bq}{\begin{equation}}
\newcommand{\eq}{\end{equation}}

\newcommand{\lt}{\left}
\newcommand{\rt}{\right}

\newcommand{\pa}{\partial}

\newcommand{\intr}{\int_{\R^d}}

\newcommand{\calP}{\mathcal P}

\newcommand{\calS}{\mathcal S}

\newcommand{\rd}{\textnormal{d}}
\newcommand{\dx}{\textnormal{d}x}

\newcommand{\dt}{\textnormal{d}t}

\newcommand{\dy}{\textnormal{d}y}

\newcommand{\ds}{\textnormal{d}s}

\makeatletter
\def\moverlay{\mathpalette\mov@rlay}
\def\mov@rlay#1#2{\leavevmode\vtop{%
   \baselineskip\z@skip \lineskiplimit-\maxdimen
   \ialign{\hfil$\m@th#1##$\hfil\cr#2\crcr}}}
\newcommand{\charfusion}[3][\mathord]{
    #1{\ifx#1\mathop\vphantom{#2}\fi
        \mathpalette\mov@rlay{#2\cr#3}
      }
    \ifx#1\mathop\expandafter\displaylimits\fi}
\makeatother

\newcommand{\supp}{{\rm supp}}

\newcommand{\tr}{{\rm tr}}

\newcommand{\diag}{{\rm diag}}
\newcommand{\cof}{{\rm cof}}

\newcommand{\nabS}{\nabla_{\!\!{}_S}}
\newcommand{\nabA}{\nabla_{\!\!{}_A}}
\newcommand{\lamin}{\lambda_{\textnormal{min}}}
\newcommand{\rdv}{\textnormal{osc}\{u_0\}}
\renewcommand{\geq}{\geqslant}
\renewcommand{\ge}{\geqslant}
\renewcommand{\leq}{\leqslant}
\renewcommand{\le}{\leqslant}

\DeclareMathOperator*{\esssup}{ess\,sup}

\begin{document}
\allowdisplaybreaks

\date{\today}

 \subjclass[2020]{35Q35, 35B44, 76N10,35L67.}
\keywords{Euler--alignment system, Lagrangian degeneracy, finite-time breakdown.}

\begin{abstract}  
We study finite-time breakdown of classical solutions to the Euler--alignment system through the degeneration of the associated Lagrangian flow. This approach allows us to characterize singularity formation in terms of the loss of local invertibility of the flow and the resulting concentration of density along characteristics. For the case of constant communication kernels, we derive an explicit formula for the flow and obtain an exact pointwise breakdown criterion in arbitrary dimension. In two dimensions, this criterion admits a closed-form reformulation in terms of the symmetric part of the initial velocity gradient and the initial vorticity. For general non-constant kernels, we derive sufficient conditions for finite-time degeneracy by combining a leading compressive mechanism with perturbative control of the nonlocal remainder. These conditions provide quantitative supercritical breakdown criteria in arbitrary dimension, complementing the existing subcritical global-regularity theory for multidimensional Euler--alignment systems.
\end{abstract}

\maketitle \centerline{\date}

\tableofcontents

%
%
%
%
\section{Introduction}
We consider the Euler--alignment system
\begin{align}\label{EA}
\begin{aligned}
&\pa_t \rho + \nabla \cdot (\rho u) = 0, \quad t>0, \ x \in \R^d, \cr
&\pa_t (\rho u) + \nabla \cdot (\rho u \otimes u) = \kappa\rho \intr \phi(x-y) (u(y) - u(x))\rho(y)\,\dy
\end{aligned}
\end{align}
which describes the collective dynamics of interacting agents driven by nonlocal velocity alignment.
Here $\rho=\rho(t,x)$ denotes the density, $u=u(t,x)$ the velocity field, $\kappa>0$ the alignment strength, and $\phi:\R^d \to \R_{\ge 0}$ a given communication kernel.

The Euler--alignment system arises as a hydrodynamic description of collective dynamics with nonlocal velocity alignment, and is closely connected to particle, kinetic, and macroscopic models of flocking and swarming. The derivation of continuum descriptions from underlying particle systems, as well as the mathematical structure of alignment dynamics at the kinetic and hydrodynamic levels, has been extensively studied in the literature; see, for instance, \cite{CC21,CCJ21, CFP25, CH24, CJ24, CK23, FP24, FK19, HL09, HT08, KMT15} and the references therein. These works include hydrodynamic and kinetic derivations for Euler--alignment type models with nonlocal or singular interactions, as well as related pressured variants. For a broader perspective on alignment models and their continuum descriptions, including particle, kinetic, and hydrodynamic viewpoints, we refer to \cite{CCTpre, CCP17, CHL17, Shv21, Shv24, Tad21} and the references therein.

The Euler--alignment system has also been widely studied from the viewpoint of qualitative properties of solutions, including global regularity, asymptotic flocking, and singularity formation. The issue of global regularity versus finite-time singularity formation has often been analyzed through the framework of critical-threshold theory. In one and two dimensions, Eulerian critical-threshold conditions for flocking hydrodynamics with nonlocal alignment were first investigated in \cite{TT14}, where subcritical regions ensuring global regularity, as well as supercritical scenarios leading to finite-time breakdown, were identified. In one dimension, this theory was later refined in \cite{CCTT16}, where a sharp pointwise critical threshold was established for the Euler--alignment case. This yields the exact dichotomy characterized by the condition
\[
u_0'(x)\ge -(\kappa \phi * \rho_0)(x).
\]
In two dimensions, finite-time blow-up on the supercritical side was obtained in \cite{TT14} under a negative divergence condition, together with additional structural assumptions on the off-diagonal components of the initial velocity gradient. A sharper Eulerian subcritical region was subsequently obtained in \cite{HT17}, formulated in terms of the divergence and the spectral gap of the symmetric part of the velocity gradient. More recently, the existence theory in arbitrary dimension was established in \cite{Tad26}, proving global smooth solutions for a class of subcritical initial data with limited initial velocity fluctuations\footnote{Here and below, we let $\nabS u$ and $\nabA u$ denote the symmetric and respectively, anti-symmetric gradients, $\nabS u:=\frac{1}{2}(\nabla u+(\nabla u)^\top\big)$,  $\nabA u:=\frac{1}{2}(\nabla u-(\nabla u)^\top\big)$, we let  $\textnormal{osc} \{u_0\}$  denote the maximal fluctuation of the initial velocity $\displaystyle\textnormal{osc} \{u_0\} :=\max_{x,y\in \supp\,\rho_0}|u_0(x)-u_0(y)|$, and $\rd^\infty_x$ denotes the diameter of spatial support of $\rho(t,\cdot)$.} $\displaystyle\textnormal{osc} \{u_0\}\leq \kappa\frac{\phi^2(\rd^\infty_x)}{8\|\nabla\phi\|_{L^\infty}}$, 
\begin{equation}\label{eq:Eulerian-CT}
\lamin(\nabS u_0) \geq -\kappa\phi(\rd^\infty_x), \quad \rd^\infty_x:=\max_t \hspace*{-0.1cm}\max_{x,y\in \supp \rho(t,\cdot)}\hspace*{-0.4cm}|x-y|.
\end{equation}
Thus, while the subcritical theory is now available in arbitrary dimension, quantitative supercritical criteria for finite-time breakdown remain much less developed. In particular, beyond two dimensions, the existing multidimensional critical-threshold theory has focused mainly on subcritical global-regularity regimes, and corresponding supercritical blow-up criteria appear to be largely absent. This gap is one of the main motivations for the present work. Related critical-threshold phenomena have also been extensively studied for Euler--Poisson and other Euler-type systems; see, for example, \cite{BLT23, CCZ16, CT08, ELT01, LT02, LT03, LTW10, TT22, TW08, WTB12}.

More recently, further developments in the one-dimensional theory addressed refined regularity criteria, cluster formation, sticky-particle dynamics, and entropic selection principles; see \cite{CCTpre, LT23, LT24}. We also mention the geometric study of mass concentration sets for pressureless Euler--alignment systems in \cite{LLST22}. These results indicate that, especially in one dimension, the fine structure of singularity formation and continuation beyond breakdown can be analyzed rather precisely.

Beyond critical-threshold theory, the Euler--alignment system also admits a well-developed existence theory. Various local and global existence results, including multidimensional well-posedness results under suitable assumptions, have been established in \cite{Cho19,CJ24,HKK14, HT17}. Long-time behavior and asymptotic flocking for Euler--alignment type models have likewise been studied in a number of works; see, for instance, \cite{CCKTpre,ST20,ST21,Tad23}.  

In this paper, we study the finite-time breakdown from a genuinely Lagrangian viewpoint. Rather than deriving closed differential inequalities for Eulerian quantities such as the divergence, the spectral gap, or the lower eigenvalue of the symmetric velocity gradient, we analyze the associated flow map and use the vanishing of its Jacobian determinant as the breakdown mechanism. More precisely, if $\eta(t,\cdot)$ denotes the associated Lagrangian flow, then
\[
J(t,x):=\det \nabla \eta(t,x)
\]
measures the local deformation of the flow. The vanishing of $J$ signals a geometric collapse of characteristics and directly leads to density concentration through
\[
\rho(t,\eta(t,x))=\frac{\rho_0(x)}{J(t,x)}.
\]
Mass concentration  is realized in the formation of Dirac mass  driven at points where  $\nabla\cdot u(t,x_*)\stackrel{t\rightarrow t_*-}{\longrightarrow} -\infty$. This provides a natural criterion for finite-time $C^1$ breakdown in the compressible regime. This geometric viewpoint is also closely related to the breakdown mechanism for pressureless gas dynamics, where the loss of invertibility of the characteristic map and the resulting density concentration play a central role; see, for instance, \cite{BCH25_2, BCH25, GS97, Pou99}.

In the current work, we derive explicit and quantitative conditions ensuring finite-time Lagrangian degeneracy. For the constant communication kernel, the alignment dynamics is simplified to damping and the Lagrangian flow admits an explicit representation, which allows us to obtain an exact pointwise criterion for the loss of local invertibility in arbitrary dimension in terms of the spectrum of the initial velocity gradient. This recovers the same critical threshold obtained by the Eulerian-based spectral dynamics \cite[\S4]{LT02}.
In the particular case of two space dimensions, this criterion can be further reformulated in closed form using the eigenvalues of the symmetric part of $\nabla u_0$ and the initial vorticity, yielding a particularly explicit description of the supercritical regime.

For general communication kernels, no explicit formula for the flow gradient is available. In this case, we decompose the flow gradient into a leading compressive part generated by the symmetric component of the initial velocity gradient and a perturbative remainder reflecting the spatial variability of the interaction kernel. This leads to quantitative, though non-sharp, sufficient conditions for finite-time Lagrangian degeneracy.   A key feature of the argument is that a sufficiently strong compressive direction in the initial data drives the Jacobian to zero before the rotational and nonlocal perturbative effects become dominant.

More precisely, our main supercritical criterion for general communication kernels can be stated as follows. It gives a quantitative finite-time degeneracy condition driven by a sufficiently negative eigenvalue of the symmetric part of the initial velocity gradient.

\begin{theorem}\label{thm:gen_phi_d}
Let $d\ge2$ and consider the Lagrange--alignment system \eqref{LA} with $\rho_0\in C^0\cap\calP(\R^d)$, $\phi\in W^{1,\infty}(\R^d;\R_{\ge0})$, and $u_0\in C^1\cap W^{1,\infty}(\R^d;\R^d)$. For $x\in\supp\,\rho_0$, let $\mu_1(x)\le\cdots\le\mu_d(x)$ be the eigenvalues of $\nabS u_0(x)$. Assume that there exist $x_*\in\supp\rho_0$, $\delta\in(0,1]$, and $t_1>0$ such that 
\[
-\mu_1(x_*) >\kappa(1+\delta)\|\phi\|_{L^\infty}, \quad \alpha_\infty(t_1) :=\frac{1-e^{-\kappa\|\phi\|_{L^\infty}t_1}} {\kappa\|\phi\|_{L^\infty}} =\frac{1+\delta}{-\mu_1(x_*)},
\]
and  the following perturbative smallness condition holds
\bq\label{eq:small_d_alpha}
\begin{split}
t_1\|\nabA u_0(x_*)\| +& \beta\lt(1+\|\nabla u_0\|_{L^\infty}\rt)t_1^2 e^{C_0t_1} <\delta,
 \end{split}
\eq
where $ \beta:=\kappa\|\nabla\phi\|_{L^\infty}\rdv$ and $C_0:=\max\{1,\beta\}$.\newline
Suppose further that
\bq\label{eq:sep_gen_d_alpha}
\min_{\alpha\in[\alpha_\infty(t_1),\,t_1]} |1+\alpha\mu_j(x_*)|\ge\delta, \quad j=1,\dots,d,
\eq
and that the number $r_*$ is odd, where
\[
 r_*:=\#\{j:1+\alpha_\infty(t_1)\mu_j(x_*)<0\}.
 \]
Then there exists $t_*=t_*(x_*)\in(0,t_1)$ such that
\[
\det\nabla\eta(t_*,x_*)=0.
\]
Consequently, the Lagrangian flow loses local invertibility in finite time along the characteristic issued from $x_*$. If in addition $\rho_0(x_*)>0$, then the density blows up along that characteristic.
\end{theorem}

In two dimensions, the determinant structure allows a sharper perturbative estimate than the dimension-uniform operator-norm argument used in Theorem \ref{thm:gen_phi_d}. This refinement is stated separately in Theorem \ref{thm:2d_gen_phi}. This two-dimensional refinement is also distinct from divergence-based blow-up criteria; see Remark \ref{rem:2d_CT_blowup_comp}.

\begin{remark}[Role of the margin parameter]
The parameter $\delta>0$ measures the spectral margin of the leading symmetric deformation after it crosses the algebraic degeneracy threshold. To see this, recall that the most compressive eigenvalue is $\mu_1(x_*)<0$. For the leading symmetric part $I+\alpha \nabS u_0(x_*)$, the corresponding factor is $1+\alpha\mu_1(x_*)$. This factor vanishes at the threshold $\alpha=-\frac1{\mu_1(x_*)}$. Since, in Theorem \ref{thm:gen_phi_d}, the reference effective time is chosen so that $\alpha_\infty(t_1)=-\frac{1+\delta}{\mu_1(x_*)}$, we have $1+\alpha_\infty(t_1)\mu_1(x_*)=-\delta$. Thus $\delta$ is not a time parameter, but a spectral margin: it measures how far the leading factor lies on the negative side after crossing zero. Keeping $\delta$ in the statement gives a quantitative criterion with an adjustable margin. A smaller value of $\delta$ allows the leading compressive deformation to be closer to the degeneracy threshold, but then the skew-symmetric and nonlocal perturbative effects must be correspondingly smaller.

For a more transparent sufficient scenario, one may fix, for instance, $\delta=1$. Then $\alpha_\infty(t_1) =-\frac{2}{\mu_1(x_*)}$, which requires in particular $ -\mu_1(x_*)>2\kappa\|\phi\|_{L^\infty}$. In this case, the perturbative condition simply says that the skew-symmetric part and the nonlocal variation scale must remain smaller than this unit spectral margin. This choice is not intended to be sharp, but it makes explicit the mechanism behind the criterion: sufficiently strong compression dominates both rotation and the spatial variability of the communication kernel.
\end{remark}
 
 \begin{remark}[On the perturbative smallness condition] 
 Condition \eqref{eq:small_d_alpha} requires that $\rdv$ and $\nabA u_0(x_*)$
  are not too large, similar to the smallness scenario sought in \cite[eq. (27)]{Tad26}
  which was assumed in order to prove global existence. 
 \end{remark}
\begin{remark}[A simple explicit supercritical scenario]
The assumptions of Theorem \ref{thm:gen_phi_d} become more transparent in the minimal sign-change configuration. Suppose that, at some point $x_*\in\supp\rho_0$,
\[
\lamin(\nabS u_0(x_*))<0, \quad \mu_j\ge0\quad (j=2,\dots,d),
\]
where $\mu_1=\lamin(\nabS u_0(x_*))\le\cdots\le\mu_d$ are the eigenvalues of $\nabS u_0(x_*)$. Set
\[
L_*:=-\lamin(\nabS u_0(x_*)).
\]
Fixing $\delta=1$, the theorem applies provided $L_*$ is sufficiently large compared with the interaction scale, the skew-symmetric part, and the nonlocal remainder. More precisely, define

\[
\Lambda_{\rm sup}(x_*) := \max\lt\{ 4\kappa\|\phi\|_{L^\infty},\, 4C_0,\, 2\lt(\|\nabA u_0(x_*)\|+\sqrt{\|\nabA u_0(x_*)\|^2+4e\beta\lt(1+\|\nabla u_0\|_{L^\infty}\rt)}\rt) \rt\}.
\]
If
\[
\lamin(\nabS u_0(x_*))<-\Lambda_{\rm sup}(x_*),
\]
then all assumptions of Theorem \ref{thm:gen_phi_d} are satisfied, and hence the Lagrangian flow degenerates at $x_*$ in finite time. This condition is not intended to be sharp. It only records that, under a simple spectral configuration of $\nabS u_0(x_*)$, sufficiently strong compression in one direction dominates the skew-symmetric and nonlocal perturbative effects. 
\end{remark}

\begin{remark}[Comparison with the subcritical threshold]
The preceding criterion should be compared with the subcritical global-regularity condition \eqref{eq:Eulerian-CT}, which requires a pointwise lower bound of the form
\[
\lamin(\nabS u_0)\ge -\kappa\phi(\rd_x^\infty),
\]
under a suitable smallness assumption on the initial velocity fluctuations. Theorem \ref{thm:gen_phi_d} gives a complementary statement on the supercritical side. In the simple scenario above, finite-time Lagrangian degeneracy follows if, at some point $x_*$,
\[
\lamin(\nabS u_0(x_*))<-\Lambda_{\rm sup}(x_*),
\]
where $\Lambda_{\rm sup}(x_*)$ is an explicit quantity controlling the interaction strength, the skew-symmetric part, and the nonlocal remainder.

Thus the comparison is not a sharp converse to \eqref{eq:Eulerian-CT}. There is a gap between the subcritical lower bound and our sufficient blow-up threshold, as expected from the perturbative nature of the argument. Nevertheless, the result gives a quantitative blow-up criterion on the opposite side of the known subcritical regime and therefore complements the global existence theory for subcritical initial data.
\end{remark}
 
The Lagrangian viewpoint offers an alternative formulation of the critical-threshold mechanism, especially in the constant-kernel case where it recovers the Eulerian spectral criterion of \cite{LT02}. Rather than tracking the spectral dynamics associated with Eulerian quantities such as the divergence, spectral gap, or $\lamin(\nabS u)$, we focus on the geometric mechanism of breakdown through the loss of local invertibility of the Lagrangian flow. In particular, while the multidimensional theory beyond two dimensions has so far concentrated mainly on subcritical global regularity and related qualitative properties, our results reveal a higher-dimensional supercritical mechanism for finite-time Lagrangian degeneracy. We also note that our two-dimensional analysis is consistent with the general observation that rotation may inhibit singularity formation: in our framework, the antisymmetric part of the initial velocity gradient acts against the formation of a real compressive eigenvalue and hence against the loss of local invertibility of the flow; see also \cite{HT17,LT04}.

The rest of this paper is organized as follows. In Section \ref{sec:Lag}, we present the Lagrangian formulation of the Euler--alignment system and relate the loss of regularity to the degeneration of the flow Jacobian. Section \ref{sec:const} is devoted to the constant communication kernel case. We derive an explicit representation of the flow, prove an exact pointwise breakdown criterion in arbitrary dimension, analyze the corresponding subcritical regime, and obtain a sharp closed-form characterization in two dimensions together with higher-dimensional sufficient supercritical conditions. In Section \ref{sec:non_const}, we study the non-constant communication kernel case. We first establish an effective-time decomposition of the flow gradient and bounds on the perturbative remainder, and then use this framework to derive sufficient conditions for finite-time Lagrangian degeneracy in two and higher dimensions.
 
%
%
%
%
%
 \section{Lagrangian formulation}\label{sec:Lag}
Assume that \eqref{EA} admits sufficiently regular solutions so that the associated Lagrangian flow map is well defined. Let $\rho_0\in C^0\cap\calP(\R^d)$ denote the initial density. We introduce the Lagrangian--alignment system
\begin{align}\label{LA}
\begin{aligned}
&\pa_t \eta(t,x) = u(t,\eta(t,x)) =: v(t,x), \quad t>0, \ x \in \supp(\rho_0), \cr
&\pa_t v(t,x)  = \kappa  \intr \phi(\eta(t,x)-\eta(t,y)) (v(t,y) - v(t,x))\,\rho_0(\dy),
\end{aligned}
\end{align}
subject to the initial condition
\[
(\eta(0,x), v(0,x)) = (x, u_0(x)), \quad x \in \supp(\rho_0).
\]

The regularity of Eulerian solutions is closely related to the invertibility of the flow map $x\mapsto \eta(t,x)$. Differentiating \eqref{LA} with respect to the Lagrangian variable $x$, we obtain 
\[
\pa_t \nabla \eta(t,x) = (\nabla u)(t,\eta(t,x))  \nabla \eta(t,x).
\]

Let
\[
J(t,x):=\det\nabla\eta(t,x)
\]
denote the Jacobian determinant of the flow. By Jacobi's formula, $J$ satisfies
\[
\pa_t J(t,x) = (\nabla \cdot u)(t,\eta(t,x)) J(t,x), \quad J(0,x) = 1,
\]
and Gr\"onwall's lemma gives
\[
J(t,x) = \exp\lt(\int_0^t  (\nabla \cdot u)(s,\eta(s,x))\,\ds\rt).
\]
In particular, since $|\nabla\cdot u|\le \|\nabla u\|_{L^\infty}$, we obtain the
quantitative lower bound
\bq\label{Jac_lower}
J(t,x)\ge \exp \lt(-\int_0^t \|\nabla u(s)\|_{L^\infty(\R^d)}\,\ds\rt), \quad t\ge0.
\eq

Thus, as long as $\nabla u\in L^1(0,T;L^\infty)$, the flow map remains a $C^1$-diffeomorphism
and $J(t,x)>0$ on $[0,T)\times\R^d$. In particular, in the incompressible case
($\nabla\cdot u=0$), one has $J\equiv1$, and the flow never degenerates.

For the compressible Euler--alignment system, the possible vanishing of $J$ provides a natural Lagrangian breakdown criterion. Indeed, if $J(t_*,x_*)=0$ for some $(t_*,x_*)$, then \eqref{Jac_lower} implies
\[
\int_0^{t_*}\|\nabla u(s)\|_{L^\infty(\R^d)}\,\ds = +\infty,
\]
so the  BKM-like continuation criterion  \cite{BKM84} associated with classical $W^{1,\infty}$ regularity class  necessarily fails at or before $t_*$. Moreover, along characteristics, the continuity equation gives
\[
\rho(t,\eta(t,x))=\frac{\rho_0(x)}{J(t,x)},
\]
and therefore, whenever $\rho_0(x_*)>0$, the degeneracy $J(t_*,x_*)=0$ yields
\[
\lim_{t\to t_*-}\rho(t,\eta(t,x_*))=+\infty.
\]

Collecting the above observations, we obtain the following consequence.

\begin{proposition} 
Let $(\rho,u)$ be a classical solution to the Euler--alignment system \eqref{EA} on $[0,T]$ such that
\[
\rho\in C([0,T];\calP \cap L^\infty(\R^d)), \quad u\in C([0,T];\dot W^{1,\infty}(\R^d)).
\]
Let $\eta$ be the associated flow and $J(t,x)=\det\nabla\eta(t,x)$. Then
\[
J(t,x)>0 \quad \text{for all } t\in[0,T],\ x\in\supp(\rho_0).
\]
Consequently, if $J(t_*,x_*)=0$ for some $x_*\in\supp(\rho_0)$, then the solution cannot belong to the above class up to time $t_*$, and in particular the lifespan in this class is finite.
\end{proposition}

\begin{proof}
Since $u\in C([0,T];\dot W^{1,\infty}(\R^d))$, we have $\nabla u\in L^1(0,T;L^\infty(\R^d))$, and thus \eqref{Jac_lower} implies $J(t,x)>0$ for all $t\in[0,T]$ and $x\in\supp(\rho_0)$. If $J(t_*,x_*)=0$, then necessarily $\int_0^{t_*}\|\nabla u(s)\|_{L^\infty}\,\ds=+\infty$, which is incompatible with the assumed regularity.
\end{proof}

%
%
%
%
%
\section{Constant communication kernel}\label{sec:const}

We begin with the constant communication kernel $\phi\equiv1$, which serves as a model case where the Lagrangian dynamics can be solved explicitly. This allows us to identify the breakdown mechanism in a sharp and transparent way, and also provides a useful benchmark for the non-constant kernel analysis developed later.  In this case, the Lagrangian system \eqref{LA} reduces to
 \begin{align}\label{LA:const}
\begin{aligned}
&\pa_t \eta(t,x) = v(t,x), \quad t>0, \ x \in \supp(\rho_0), \cr
&\pa_t v(t,x)  = \kappa \intr   (v(t,y) - v(t,x))\,\rho_0(\dy),
\end{aligned}
\end{align}
subject to the initial condition
\[
(\eta(0,x), v(0,x)) = (x, u_0(x)), \quad x \in \supp(\rho_0).
\]

%
%
%
%
%
\subsection{Explicit $C^1$ breakdown and regularity with constant kernels} 

For the constant communication kernel $\phi\equiv1$, the Lagrangian system admits an explicit solution. As a consequence, one can characterize precisely when the flow loses local invertibility and, conversely, when uniform regularity persists.

Since the total momentum is conserved, we obtain
\[
\intr v(t,x)\,\rho_0(\dx) = \intr u_0(x)\,\rho_0(\dx) =: \bar v_0,
\]
and thus $v$ satisfies
\[
\pa_t v(t,x) = \kappa (\bar v_0 - v).
\]
Then, by solving the momentum equation, we get
\[
v(t,x) = \bar v_0 + (u_0(x)  - \bar v_0) e^{-\kappa t},
\]
and subsequently, 
\[
\eta(t,x) = x + \int_0^t v(s,x)\,\ds = x + \bar v_0 t + \alpha(t) (u_0(x)  - \bar v_0), \quad \alpha(t) := \frac{1 - e^{-\kappa t}}{\kappa}.
\]
This implies
\bq\label{expl_con}
\nabla \eta(t,x) = I_d + \alpha(t) \nabla u_0(x) \quad \text{and} \quad \det \nabla \eta(t,x) = \det \lt( I_d + \alpha(t) \nabla u_0(x)\rt).
\eq

The next proposition gives an exact criterion for the loss of local invertibility of the Lagrangian flow.

\begin{proposition} \label{prop:const-break}
Consider the system \eqref{LA:const}. Assume $\rho_0\in C^0\cap\calP(\R^d)$,  $u_0\in C^1 \cap W^{1,\infty}(\R^d;\R^d)$ and fix a point $x_*\in\supp(\rho_0)$.
The following statements are equivalent:
\begin{enumerate}
\item[\rm (i)] There exists $t>0$ such that $\det\nabla\eta(t,x_*)=0$.
\item[\rm (ii)] There exists $\alpha\in(0,1/\kappa)$ such that
\[
\det\big(I_d+\alpha\nabla u_0(x_*)\big)=0.
\]
\item[\rm (iii)] The matrix $\nabla u_0(x_*)$ has a real eigenvalue $\lambda_*<-\kappa$.
\end{enumerate}

Moreover, assume that $\nabla u_0(x_*)$ has at least one real eigenvalue below $-\kappa$ and set
\[
\lambda_{\min}(x_*):=\min\lt\{\lambda\in\sigma(\nabla u_0(x_*))\cap\R: \lambda<-\kappa\rt\}.
\]
Then the first breakdown time at $x_*$ is
\[
t_*(x_*)
= -\frac1\kappa \log\lt(1+ \frac{\kappa}{\lambda_{\min}(x_*)}\rt)\in(0,\infty),
\]
that is, $\alpha(t_*)=-1/\lambda_{\min}(x_*)$ and hence $\det\nabla\eta(t_*,x_*)=0$.
\end{proposition}

\begin{proof}
For the constant communication weight $\phi\equiv1$, the explicit formula
\eqref{expl_con} gives
\[
\nabla\eta(t,x_*)=I_d+\alpha(t)\nabla u_0(x_*), \quad \alpha(t)=\frac{1-e^{-\kappa t}}{\kappa}.
\]
The function $\alpha(t)$ is strictly increasing on $(0,\infty)$ and satisfies $\alpha(t)\in(0,1/\kappa)$.
Consequently, the condition $\det\nabla\eta(t,x_*)=0$ for some $t>0$ is equivalent to the existence of a parameter $\alpha\in(0,1/\kappa)$ such that
\[
\det\lt(I_d+\alpha\nabla u_0(x_*)\rt)=0.
\]
This shows the equivalence between {\rm(i)} and {\rm(ii)}.

Set $M_*:=\nabla u_0(x_*)$. The singularity of the matrix $I_d+\alpha M_*$ means that there exists a nonzero vector $\xi\in\R^d$ such that
\[
(I_d+\alpha M_*)\xi=0. 
\]
Equivalently, $\xi$ is an eigenvector of $M_*$ associated with the eigenvalue $-1/\alpha$. Since $\alpha>0$, this eigenvalue is real and negative. Therefore, condition {\rm(ii)} holds for some $\alpha\in(0,1/\kappa)$ if and only if $M_*$ admits a real eigenvalue $\lambda_*=-1/\alpha$ satisfying
\[
\alpha\in(0,1/\kappa) \quad \Longleftrightarrow \quad \lambda_*<-\kappa.
\]
This establishes the equivalence between {\rm(ii)} and {\rm(iii)}.

Assume now that $\nabla u_0(x_*)$ has at least one real eigenvalue $\lambda<-\kappa$, and denote by
\[
\lambda_{\min}=\min\lt\{\lambda\in\sigma(\nabla u_0(x_*))\cap\R : \lambda<-\kappa\rt\}.
\]
Define $\alpha_*:=-1/\lambda_{\min}\in(0,1/\kappa)$. Since the function $\alpha(t)$ is strictly increasing on $(0,\infty)$, there exists a unique $t_*>0$ such that
\[
\alpha(t_*)=\alpha_*.
\]
Solving this relation yields
\[
e^{-\kappa t_*}=1+\frac{\kappa}{\lambda_{\min}}\in(0,1),
\quad
t_*=-\frac1\kappa\log\lt(1+\frac{\kappa}{\lambda_{\min}}\rt).
\]
At this time, the matrix $I_d+\alpha(t_*)\nabla u_0(x_*)$ is singular, and hence
\[
\det\nabla\eta(t_*,x_*)=0.
\]
Moreover, if $\nabla u_0(x_*)$ admits several real eigenvalues below $-\kappa$, then each such eigenvalue
$\lambda$ generates a potential singular time $t_\lambda$ via $\alpha(t_\lambda)=-1/\lambda$, but the
first breakdown occurs at $t_*$ corresponding to $\lambda_{\min}$, since $\alpha(t)$ is strictly increasing.
\end{proof}

\begin{remark}[Eulerian formulation]
Proposition \ref{prop:const-break} shows that in the case of constant kernel $\phi\equiv1$, finite-time degeneration of the Lagrangian flow at a point $x_*$ can occur if and only if $\nabla u_0(x_*)$ admits a real eigenvalue below $-\kappa$. Equivalently, if
\begin{equation}\label{eq:sigma-below-kappa}
\sigma(\nabla u_0(x_*))\cap(-\infty,-\kappa)=\emptyset,
\end{equation}
then
\[
J(t,x_*)>0 \quad \text{for all } t\ge0.
\]
This coincides with the critical threshold condition derived from the  Eulerian 
dynamics, \eqref{EA}${}_2$ with $\phi\equiv1$,
\[
u_t+u\cdot\nabla_x u= \kappa(\overline{v}_0 - u),
\]
where we used the conservation of mass $\int \rho(t,\cdot)\,\rd x = 1$ and of momentum $\overline{v}_0=\int (\rho u)(t,\cdot)\,\rd x$.  Indeed, arguing along the lines of \cite[Lemma 4.1]{LT02}, it follows that the velocity gradient, $M:=\nabla u$, satisfies the matrix Ricatti equation
\[
M'+M^2 =   -\kappa M,  \quad \square':=(\partial_t +u\cdot\nabla_x)\square, 
\]
which in turn yields that  the eigenvalues, $\lambda=\lambda(M)$, satisfy the scalar Ricatti equation 
$\lambda' + \lambda^2 = -\kappa\lambda$,
whose solution is given by 
\begin{equation}\label{eq:lambda-of-M}
\lambda(t) =\frac{\lambda_0(x)e^{-\kappa t}}{1+\lambda_0(x)\kappa^{-1}(1-e^{-\kappa t})}.
\end{equation}
This solution along particle-path remains bounded from below for all time if and only if 
\[ 
\lambda_0(x) \ \text{is either complex or} \  
 \inf_{\alpha \in {\mathbb R}^d}\lambda_0(x) \geq -\kappa,
\] 
  that is, if and only if \eqref{eq:sigma-below-kappa} holds.
\end{remark}

We next refine Proposition \ref{prop:const-break} by quantifying the rate at which the Jacobian $J(t,x_*)$ vanishes as $t \to t_*-$, and hence the associated density blow-up along the characteristic issued from $x_*$.

\begin{proposition} \label{prop:const-rate}
Consider the system \eqref{LA:const}. Assume $\rho_0\in C^0\cap\calP(\R^d)$,  $u_0\in C^1\cap W^{1,\infty}(\R^d;\R^d)$ and fix $x_*\in\supp(\rho_0)$. Suppose that the first breakdown time at $x_*$ is $t_*\in(0,\infty)$, i.e.
\[
J(t,x_*)>0\ \text{for }0\le t<t_*, \quad J(t_*,x_*)=0.
\]
Let $\lambda_*:=-1/\alpha(t_*)$ be an eigenvalue of $\nabla u_0(x_*)$ with algebraic multiplicity $m\ge1$, where $\alpha(t)$ is defined as in \eqref{expl_con}. Then there exist constants $c_1,c_2>0$ and $\delta>0$ such that, for all $t\in(t_*-\delta,t_*)$,
\[
c_1 (t_*-t)^m \le J(t,x_*) \le c_2 (t_*-t)^m .
\]
In particular, if $\rho_0(x_*)>0$, then there exist $C_1,C_2>0$ such that
\[
C_1 (t_*-t)^{-m}\le \rho(t,\eta(t,x_*))\le C_2 (t_*-t)^{-m}
\]
for all $t$ sufficiently close to $t_*$.
Moreover,
\[
\int_0^t (\nabla\cdot u)(s,\eta(s,x_*))\,\ds
=
m\log(t_*-t)+O(1)
\quad (t\to t_*-).
\]
\end{proposition}

\begin{proof}
By \eqref{expl_con}, we have
\[
J(t,x_*)=\det\nabla\eta(t,x_*)=\det(I_d+\alpha(t)M_*)=:f(\alpha(t)), \quad M_*:=\nabla u_0(x_*).
\]
Let $\lambda_1,\dots,\lambda_d$ be the eigenvalues of $M_*$ counted with algebraic multiplicity.
Then
\[
f(\alpha)=\prod_{j=1}^d(1+\alpha\lambda_j).
\]
Set $\alpha_*:=\alpha(t_*)$, so that $f(\alpha_*)=0$ and $\lambda_*=-1/\alpha_*$.
Since $\lambda_*$ has algebraic multiplicity $m$, after relabeling we may assume
\[
\lambda_1=\cdots=\lambda_m=\lambda_*,
\quad
\lambda_{m+1},\dots,\lambda_d\neq \lambda_*.
\]
Thus we can factor
\[
f(\alpha)=(1+\alpha\lambda_*)^m\,g(\alpha),
\quad
g(\alpha):=\prod_{j=m+1}^d(1+\alpha\lambda_j).
\]
Since $1+\alpha_*\lambda_j\neq0$ for $j\ge m+1$, we have $g(\alpha_*)\neq0$, and by continuity there exist
$0<c_-\le c_+$ and $\delta_1>0$ such that
\[
c_-\le |g(\alpha)|\le c_+ \quad \text{whenever }|\alpha-\alpha_*|<\delta_1.
\]
Using $1+\alpha\lambda_*=\lambda_*(\alpha-\alpha_*)$, we obtain for $|\alpha-\alpha_*|<\delta_1$,
\[
|f(\alpha)|=|\lambda_*|^m\,|\alpha-\alpha_*|^m\,|g(\alpha)|.
\]
Hence, there exist $C_-,C_+>0$ such that
\[
C_-|\alpha-\alpha_*|^m\le |f(\alpha)|\le C_+|\alpha-\alpha_*|^m,
\quad |\alpha-\alpha_*|<\delta_1.
\]

Next, since $\alpha'(t)=e^{-\kappa t}>0$, the mean value theorem yields that for each $t<t_*$ there exists $\tau_t\in(t,t_*)$ such that
\[
\alpha_*-\alpha(t)=\alpha'(\tau_t)(t_*-t)=e^{-\kappa\tau_t}(t_*-t).
\]
Choosing $\delta_0>0$ sufficiently small so that $\tau_t\in(t_*-\delta_0,t_*)$ whenever $t\in(t_*-\delta_0,t_*)$, we get
\[
\frac12 e^{-\kappa t_*}(t_*-t)\le \alpha_*-\alpha(t)\le 2 e^{-\kappa t_*}(t_*-t)
\quad (t\in(t_*-\delta_0,t_*)).
\]

Since $t_*$ is the first breakdown time, $J(t,x_*)>0$ for $t<t_*$, and thus for $t$ close enough to $t_*$,
\[
J(t,x_*)=f(\alpha(t))=|f(\alpha(t))|.
\]
Combining the previous two displays gives the two-sided bound
\[
c_1 (t_*-t)^m \le J(t,x_*) \le c_2 (t_*-t)^m
\]
for all $t\in(t_*-\delta,t_*)$, with $\delta:=\min\{\delta_0,\delta_2\}$ and $\delta_2>0$ chosen so that
$|\alpha(t)-\alpha_*|<\delta_1$.

The density estimate follows immediately from
\[
\rho(t,\eta(t,x_*))=\frac{\rho_0(x_*)}{J(t,x_*)}.
\]
The logarithmic asymptotic then follows from Jacobi's formula
\[
J(t,x_*)=\exp\lt(\int_0^t (\nabla\cdot u)(s,\eta(s,x_*))\,\ds\rt).
\]
This completes the proof.
\end{proof}

\begin{remark}[Divergence blow-up]
Proposition \ref{prop:const-rate} yields
\[
\log J(t,x_*) = m\log(t_*-t)+O(1)\quad (t\to t_*-),
\]
and hence
\[
\int_0^t (\nabla\cdot u)(s,\eta(s,x_*))\,\ds \to -\infty
\quad (t\to t_*-).
\]
Although this does not by itself imply the full pointwise limit
\[
(\nabla\cdot u)(t,\eta(t,x_*))\to -\infty,
\]
it does show, by a mean-value argument on shrinking time intervals, that there exists a sequence $t_n\uparrow t_*$ such that
\[
(\nabla\cdot u)(t_n,\eta(t_n,x_*))\to -\infty.
\]
Equivalently,
\[
\liminf_{t\to t_*-}(\nabla\cdot u)(t,\eta(t,x_*))=-\infty.
\]
In the constant-kernel case this blow-up is necessarily compressive. Indeed, along a characteristic before the first degeneracy time, the eigenvalues of $\nabla u$ are explicitly given by \eqref{eq:lambda-of-M}; the positive real eigenvalues remain uniformly bounded from above, while the singularity can only occur through a real eigenvalue tending to $-\infty$. Thus the loss of regularity detected above corresponds to compression rather than expansion.
\end{remark}

We conclude this subsection by providing a complementary fact: as long as the Jacobian remains uniformly positive on $[0,T]\times K$, the flow remains uniformly invertible on $K$ and  the corresponding pulled-back Eulerian quantities remain uniformly bounded. If, in addition, the flow is injective on $K$, these bounds yield genuine Eulerian $W^{1,\infty}$ bounds on $\eta(t,K)$.

\begin{proposition} \label{prop:subcri_const_d}
Let $d\ge2$ and consider the system \eqref{LA:const}.
Assume $u_0\in C^1\cap W^{1,\infty}(\R^d;\R^d)$ and $\rho_0\in C_b\cap\calP(\R^d)$.
Set $K:=\supp\rho_0$. Fix $T>0$ and assume that  
\bq\label{eq:det_lower_d}
\inf_{(t,x)\in[0,T]\times K}\det\nabla\eta(t,x)=:c_T >0.
\eq
Then the following statements hold.

\smallskip\noindent
\rm (i) For every $t\in[0,T]$ and $x\in K$, the matrix $\nabla\eta(t,x)$ is invertible, and there exists a constant $C_T>0$ such that
\[
\sup_{(t,x)\in[0,T]\times K}\|(\nabla\eta(t,x))^{-1}\|\le C_T.
\]

\noindent
\rm (ii) The pulled-back density, velocity, and Eulerian velocity gradient satisfy
\[
\sup_{0\le t\le T} \left\|\frac{\rho_0}{J(t,\cdot)}\right\|_{L^\infty(K)} + \sup_{0\le t\le T} \|v(t,\cdot)\|_{L^\infty(K)} + \sup_{0\le t\le T} \|\nabla v(t,\cdot)(\nabla\eta(t,\cdot))^{-1}\|_{L^\infty(K)} \le C,
\]
for some constant $C>0$ depending only on $c_T$, $T$, $\kappa$, $\|u_0\|_{W^{1,\infty}}$, and $\|\rho_0\|_{L^\infty}$.

\smallskip\noindent
\rm (iii) If, in addition, $\eta(t,\cdot)|_K$ is injective for every $t\in[0,T]$, then $\eta(t,\cdot)|_K$ is a $C^1$-diffeomorphism from $K$ onto $\eta(t,K)$ for every $t\in[0,T]$.  Moreover, the corresponding Eulerian solution satisfies
\[
\sup_{0\le t\le T}\|\rho(t)\|_{L^\infty(\eta(t,K))}+\sup_{0\le t\le T}\|u(t)\|_{W^{1,\infty}(\eta(t,K))}\le C,
\]
for some constant $C>0$ depending only on $c_T$, $T$, $\kappa$, $\|u_0\|_{W^{1,\infty}}$, and $\|\rho_0\|_{L^\infty\cap\calP}$.
\end{proposition}

\begin{proof}
We use the explicit formulas for the constant-kernel dynamics:
\[
v(t,x)=\bar v_0+\lt(u_0(x)-\bar v_0\rt)e^{-\kappa t},
\quad
\nabla v(t,x)=e^{-\kappa t}\nabla u_0(x),
\]
\[
\eta(t,x)=x+\alpha(t)\lt(u_0(x)-\bar v_0\rt),
\quad
\nabla\eta(t,x)=I_d + \alpha(t)\nabla u_0(x),
\quad
\alpha(t)=\frac{1-e^{-\kappa t}}{\kappa}.
\]
In particular,
\[
J(t,x)=\det\nabla\eta(t,x)=\det(I_d + \alpha(t)\nabla u_0(x)).
\]
Assumption \eqref{eq:det_lower_d} yields $J(t,x)\ge c_T$ for all $(t,x)\in[0,T]\times K$, thus $\nabla\eta(t,x)$ is invertible there. Moreover, since $\alpha(t)\le 1/\kappa$ and
$\nabla u_0\in L^\infty$, we have
\bq\label{eq:bdd_eta}
\|\nabla\eta(t,x)\|=\|I_d + \alpha(t)\nabla u_0(x)\|
\le 1+\frac{1}{\kappa}\|\nabla u_0\|_{L^\infty} =: N_T \quad \text{for all } (t,x)\in[0,T]\times K.
\eq
Hence,  the standard bound $\displaystyle 
\|A^{-1}\|\le \frac{\|A\|^{d-1}}{|\det A|}$
for $A=\nabla\eta(t,x)$, together with \eqref{eq:bdd_eta} proves (i),
\[
\|(\nabla\eta(t,x))^{-1}\|\le \frac{N_T^{d-1}}{J(t,x)}\leq  \frac{N_T^{d-1}}{c_T}=:C_T\quad \text{for all }(t,x)\in[0,T]\times K.
\]
Next, define the pulled-back density and the pulled-back Eulerian velocity gradient by
\[
\rho^\sharp(t,x):=\frac{\rho_0(x)}{J(t,x)},\quad M^\sharp(t,x):=\nabla v(t,x)(\nabla\eta(t,x))^{-1}.
\]
Then, by the explicit formula for $v$,
\[
\sup_{0\le t\le T}\|v(t,\cdot)\|_{L^\infty(K)} \le |\bar v_0|+\|u_0\|_{L^\infty} \le 2\|u_0\|_{L^\infty},
\]
where we used $\rho_0\in\calP(\R^d)$. Moreover, by (i),
\bq\label{eq:pullback-grad-bd}
\|M^\sharp(t,x)\| \le \|\nabla v(t,x)\| \|(\nabla\eta(t,x))^{-1}\| \le e^{-\kappa t}\|\nabla u_0\|_{L^\infty} C_T \le \|\nabla u_0\|_{L^\infty} C_T .
\eq
Finally, since $J(t,x)\ge c_T$, we have
\[
\|\rho^\sharp(t,\cdot)\|_{L^\infty(K)} = \lt\|\frac{\rho_0}{J(t,\cdot)}\rt\|_{L^\infty(K)} \le \frac{\|\rho_0\|_{L^\infty}}{c_T}.
\]
Collecting these estimates proves (ii).

For (iii), fix $t\in[0,T]$. By (i), the derivative $\nabla\eta(t,x)$ is invertible at each $x\in K$, so the inverse function theorem shows that $\eta(t,\cdot)$ is a local $C^1$-diffeomorphism near every point of $K$. If, in addition, the restriction $\eta(t,\cdot)|_K$ is injective, then its inverse on $\eta(t,K)$ is well defined and locally $C^1$, hence $\eta(t,\cdot)|_K$ is a $C^1$-diffeomorphism from $K$ onto $\eta(t,K)$.
 Writing $\xi_t:=(\eta(t,\cdot)|_K)^{-1}$, the Eulerian solution on $\eta(t,K)$
is given by
\[
u(t,z)=v(t,\xi_t(z)),\quad \rho(t,z)=\rho^\sharp(t,\xi_t(z)),\quad z\in\eta(t,K).
\]
Furthermore,
\[
\nabla u(t,z)=M^\sharp(t,\xi_t(z)).
\]
Hence, the estimates in (ii) imply
\[
\sup_{0\le t\le T}\|\rho(t)\|_{L^\infty(\eta(t,K))} + \sup_{0\le t\le T}\|u(t)\|_{W^{1,\infty}(\eta(t,K))} \le C.
\]
This completes the proof.
\end{proof}

\begin{remark}
While the  Eulerian formulation yields that for sub-critical initial data,
the spectral gap $\sigma(\nabla u(t,\cdot))\backslash (-\infty,-\kappa)$ remains  lower bounded,  the Lagrangian formulation provides  the pulled-back uniform gradient bound \eqref{eq:pullback-grad-bd}. Under the additional injectivity
assumption on $\eta(t,\cdot)|_K$, this becomes a genuine Eulerian $L^\infty(\eta(t,K))$ bound on $\nabla u(t,\cdot)$.
\end{remark}

\begin{remark}[Uniform subcriticality: role of compact support]
The compactness of $K=\supp\rho_0$ is not essential for the uniform non-degeneracy estimates in Proposition \ref{prop:subcri_const_d}. The argument only relies on the existence of a strictly positive lower bound for 
\[
\det\nabla\eta(t,x)=\det(I_d + \alpha(t)\nabla u_0(x))
\]
 on $[0,T]\times K$.

However, when $K$ is compact, the subcritical condition can be formulated in a more concrete way. Indeed, since $(t,x)\mapsto \det\nabla\eta(t,x)$ is continuous and $\alpha(t)$ takes values in the compact interval $[0,\alpha(T)]$, the pointwise condition
\[
\det\nabla\eta(t,x)>0
\quad \text{for all } (t,x)\in[0,T]\times K
\]
automatically implies the uniform lower bound
\[
\inf_{(t,x)\in[0,T]\times K}\det\nabla\eta(t,x)>0.
\]
In this sense, compactness of the support allows one to replace the explicit uniform positivity assumption \eqref{eq:det_lower_d} by a pointwise non-degeneracy condition.
\end{remark}

%
%
%
%
%

\subsection{A two-dimensional strain--vorticity reformulation}
 
Although Proposition \ref{prop:const-break} gives an exact spectral criterion in all dimensions, it is often useful to formulate the condition in terms of more directly interpretable components of the initial velocity gradient. We now provide such a reformulation in two dimensions, using the symmetric and antisymmetric parts of $\nabla u_0$. The purpose is not to introduce a different breakdown mechanism, but to express the same spectral threshold in an explicit initial-configuration form and to make transparent the role of vorticity. Such strain--vorticity or spectral-gap formulations are natural in two-dimensional critical-threshold theory; see, for instance, \cite{LT04}. 
We recall that Proposition \ref{prop:const-break} characterizes finite-time Lagrangian degeneracy by the existence of a real eigenvalue of $\nabla u_0(x)$ below $-\kappa$. In general, this condition cannot be expressed solely in terms of the eigenvalues of $\nabS u_0(x)$, since the antisymmetric part may eliminate real spectrum altogether.

In two space dimensions, however, $\nabA u_0$ is completely determined by the scalar vorticity $\omega_0=\nabla\times u_0$. As a consequence, the spectrum of $\nabla u_0(x)$ can be described explicitly in terms of the eigenvalues of $\nabS u_0(x)$ and $\omega_0(x)$. This yields a closed-form and pointwise sharp characterization of the supercritical regime.

\begin{corollary}\label{cor:2d_sharp}
Let $d=2$ and consider the  system \eqref{LA:const}. Assume $\rho_0\in C^0\cap\calP(\R^2)$ and $u_0\in C^1\cap W^{1,\infty}(\R^2;\R^2)$. For $x\in\supp(\rho_0)$, let $\mu_1(x)\le \mu_2(x)$ be the eigenvalues of $\nabS u_0(x)$ and set $\omega_0(x):=(\nabla\times u_0)(x)$. Define
\[
\lambda_-(\mu_1,\mu_2,\omega_0)
:=\frac{\mu_1+\mu_2-\sqrt{(\mu_1-\mu_2)^2-\omega_0^2}}{2},
\]
whenever $(\mu_1-\mu_2)^2\ge \omega_0^2$.
We also define the \emph{supercritical set} in the parameter space by
\[
\calS_{{\rm sup}}
:=\lt\{(\mu_1,\mu_2,\omega_0)\in\R^3:  
(\mu_1-\mu_2)^2\ge \omega_0^2  \text{ and } 
\lambda_-(\mu_1,\mu_2,\omega_0)<-\kappa\rt\}.
\]
Then the following are equivalent:
\begin{enumerate}
\item[\rm (i)] There exists $x_* \in \supp(\rho_0)$ and $t>0$ such that $\det\nabla\eta(t,x_*)=0$.
\item[\rm (ii)] There exists $x_*\in\supp\rho_0$ such that $(\mu_1(x_*),\mu_2(x_*),\omega_0(x_*))\in \calS_{{\rm sup}}$.
\end{enumerate}
In this case, the first breakdown time along the characteristic issued from such an $x_*$ is
\[
t_*=-\frac1\kappa\log\lt(1+\frac{\kappa}{\lambda_-(x_*)}\rt)\quad
\lambda_-(x_*):=\lambda_-\big(\mu_1(x_*),\mu_2(x_*),\omega_0(x_*)\big).
\]
\end{corollary}

\begin{remark}[Rotation as a stabilizing effect] Corollary \ref{cor:2d_sharp} shows explicitly that the antisymmetric part of the initial velocity gradient acts against finite-time degeneracy. Indeed, the initial vorticity enters through the discriminant
\[
(\mu_1-\mu_2)^2-\omega_0^2,
\]
so that a large rotational component may prevent the appearance of a real eigenvalue of $\nabla u_0(x_*)$ below $-\kappa$. In this sense, rotation has a stabilizing effect on the Lagrangian flow and may obstruct the loss of local invertibility. This is consistent with related observations in two-dimensional rotational flow models and Eulerian critical-threshold theory; see, for instance, \cite{HT17,LT04,Tad23}.
\end{remark}
\begin{proof}  Fix $x_*\in\supp\rho_0$. The pointwise equivalence follows from Proposition \ref{prop:const-break}. Taking the existence of such a point $x_*$ gives the equivalence between {\rm(i)} and {\rm(ii)}.

In $d=2$, the skew-symmetric part $A_*:=\nabA u_0(x_*)$ has the form
\[
A_*= \frac12\begin{pmatrix}0&-\omega_0(x_*)\\ \omega_0(x_*)&0\end{pmatrix},  
\]
while $S_*=\nabS u_0(x_*)$ is symmetric with eigenvalues $\mu_1\le\mu_2$. Choosing an orthonormal eigenbasis of $S_*$, we may assume without loss of generality that
\[
S_*=\begin{pmatrix}\mu_1(x_*) & 0 \\ 0 & \mu_2(x_*) \end{pmatrix}.
\]
Since trace and determinant are invariant under orthogonal similarity transformations, the quantities $\tr M_*$ and $\det M_*$ computed below do not depend on the chosen eigenbasis of $S_*$.

In this basis, the full gradient matrix reads
\[
M_*=S_*+A_*= \begin{pmatrix}
\mu_1(x_*) & - \frac{\omega_0(x_*)}2\\
\frac{\omega_0(x_*)}2 & \mu_2(x_*)
\end{pmatrix}.
\]

From this explicit $2\times2$ representation, we compute the basic invariants
\[
\tr M_*=\mu_1(x_*)+\mu_2(x_*), \quad \det M_* =\mu_1(x_*)\mu_2(x_*) + \frac{\omega_0(x_*)^2}4.
\]
Consequently, the characteristic polynomial of $M_*$,
\[
p(\lambda)=\det(\lambda I_2-M_*) =\lambda^2-(\tr M_*)\lambda+\det M_*,
\]
has discriminant
\begin{align*}
(\tr M_*)^2-4\det M_* &=(\mu_1(x_*)+\mu_2(x_*))^2-4\lt(\mu_1(x_*)\mu_2(x_*) + \frac{\omega_0(x_*)^2}4\rt)\cr
&=(\mu_1(x_*)-\mu_2(x_*))^2- \omega_0(x_*)^2.
\end{align*}
Therefore, $M_*$ has a real eigenvalue below $-\kappa$ if and only if
\[
(\mu_1(x_*),\mu_2(x_*),\omega_0(x_*))\in \calS_{\rm sup}.
\]
By Proposition \ref{prop:const-break}, this is equivalent to the existence of $t>0$ such that
\[
\det\nabla\eta(t,x_*)=0.
\]
The formula for the first breakdown time then follows from the general expression
\[
t_*=-\frac1\kappa\log\lt(1+\frac{\kappa}{\lambda_-(x_*)}\rt),
\]
again given by Proposition \ref{prop:const-break}. This completes the proof.
\end{proof}

%
%
%
%
%

\subsection{A perturbative criterion in terms of the symmetric part}

The two-dimensional result above provides a complete algebraic characterization of breakdown in terms of the symmetric and skew-symmetric parts of the velocity gradient. In higher dimensions, however, a comparably simple closed-form description in terms of these two components is not available in general, since the spectral structure of nonsymmetric matrices becomes considerably more involved. Nevertheless, it is still possible to derive robust sufficient conditions for finite-time degeneration of the Lagrangian flow by treating the skew-symmetric part as a perturbation of the symmetric one and detecting a sign change in the determinant of $I_d+\alpha\nabla u_0(x_*)$. The idea is to choose a reference time parameter $\alpha_1\in(0,1/\kappa)$ for which the diagonal part $I_d+\alpha_1 S_*$ already has a negative determinant, and then show that this sign persists under a sufficiently small skew-symmetric perturbation. This is a sufficient consequence of Proposition \ref{prop:const-break} and serves as preparation for the non-constant kernel analysis below.

\begin{corollary} \label{cor:gen_sup_const}
Let $d \geq 3$ and consider the system \eqref{LA:const}. Assume $\rho_0\in C^0\cap\calP(\R^d)$ and $u_0\in C^1\cap W^{1,\infty}(\R^d;\R^d)$. Suppose that there exist $x_*\in\supp(\rho_0)$, $\alpha_1\in(0,\frac1\kappa)$, and $\delta>0$ such that, setting
\[
M_*:=\nabla u_0(x_*),\quad S_*:=\nabS u_0(x_*),\quad A_*:=\nabA u_0(x_*),
\]
and denoting by $\mu_1\le\cdots\le\mu_d$ the eigenvalues of $S_*$, one has
\bq\label{eq:sep0}
|1+\alpha_1\mu_j|\ge \delta \quad\text{for all }j=1,\dots,d,
\eq
and the number
\[
r:=\#\{j\in\{1,\dots,d\}:1+\alpha_1\mu_j<0\}
\]
is odd. Assume moreover the smallness condition
\bq\label{eq:sma0}
\alpha_1\|A_*\|<\delta.
\eq
Then there exists $\alpha_*\in(0,\alpha_1)$ such that
\[
\det(I_d+\alpha_*M_*)=0.
\]
Consequently, there exists
\[
t_*\in\lt(0,-\frac1\kappa\log(1-\kappa\alpha_1)\rt)
\]
such that
\[
\det\nabla\eta(t_*,x_*)=0.
\]
In particular, the Lagrangian flow degenerates at $x_*$ in finite time.
\end{corollary}

\begin{remark}[A concrete sufficient scenario]
Assume that the symmetric part $\nabS u_0(x_*)$ has exactly one negative eigenvalue and that
\[
\lambda_{\min} (\nabS u_0(x_*)) < -(\kappa+\|\nabA u_0(x_*)\|), \quad \sigma(\nabS u_0(x_*))\setminus\{\lambda_{\min}(\nabS u_0(x_*))\}\subset[0,\infty).
\]
Then one may choose $\alpha_1\in(0,1/\kappa)$ close enough to $1/\kappa$ so that
\[
1+\alpha_1\lambda_{\min}(\nabS u_0(x_*))<-\alpha_1\|\nabA u_0(x_*)\|,
\]
while
\[
1+\alpha_1\mu_j>0
\quad\text{for all other eigenvalues }\mu_j\ge0.
\]
Hence \eqref{eq:sep0} holds for some $\delta>0$, exactly one factor $1+\alpha_1\mu_j$ is negative, and \eqref{eq:sma0} is satisfied as well. Therefore Corollary \ref{cor:gen_sup_const} applies, and the associated flow degenerates at $x_*$ in finite time.
\end{remark}

\begin{proof}[Proof of Corollary \ref{cor:gen_sup_const}]
We proceed in several steps. 

\medskip
\noindent
\textit{Step 1: Reduction to an algebraic root in $\alpha$.}
We first recall
\[
\nabla\eta(t,x_*)=I_d+\alpha(t)M_*, \quad \alpha(t)=\frac{1-e^{-\kappa t}}{\kappa}\in \lt[0,\frac1\kappa\rt).
\]
Let $f(\alpha):=\det(I_d+\alpha M_*)$ for $\alpha\in[0,1/\kappa)$. Then $\det\nabla\eta(t,x_*)=f(\alpha(t))$, and $f$ is continuous with $f(0)=1>0$. Hence it suffices to prove $f(\alpha_1)<0$, since then the intermediate value theorem yields $\alpha_*\in(0,\alpha_1)$ with $f(\alpha_*)=0$, and the monotonicity of $\alpha(t)$ provides $t_*>0$ such that $\alpha(t_*)=\alpha_*$.

\medskip
\noindent
\textit{Step 2: Diagonalization of $S_*$ and sign of the unperturbed determinant.}
Since $S_*$ is symmetric, there exists an orthogonal matrix $Q\in O(d)$ such that
\[
Q^\top S_*Q=\Lambda:=\diag(\mu_1,\dots,\mu_d).
\]
Set $\widetilde A_*:=Q^\top A_*Q$, which remains skew-symmetric, and note that $\|\widetilde A_*\|=\|A_*\|$.
By invariance of the determinant under orthogonal similarity transformations,
\[
f(\alpha_1)=\det \lt(I_d+\alpha_1(\Lambda+\widetilde A_*)\rt).
\]
Define the diagonal matrix
\[
D:=I_d+\alpha_1\Lambda=\diag(d_1,\dots,d_d), \quad d_j:=1+\alpha_1\mu_j,
\]
so that
\[
f(\alpha_1)=\det(D+\alpha_1\widetilde A_*).
\]
By the separation condition \eqref{eq:sep0}, we have $d_j\neq0$ for all $j$ and hence $D$ is invertible.
Moreover, since exactly $r$ of the $d_j$ are negative and $r$ is odd, it follows that
\[
\det D=\prod_{j=1}^d d_j= -\prod_{j=1}^d |d_j|< 0.
\]

\medskip
\noindent
\textit{Step 3: Stability of the sign under a small skew perturbation.}
We write
\[
f(\alpha_1)=\det(D+\alpha_1\widetilde A_*)
=\det D \cdot \det(I_d+X),
\quad
X:=D^{-1}(\alpha_1\widetilde A_*).
\]
Since $\|D^{-1}\|\le 1/\delta$, we obtain
\[
\|X\|\le \frac{\alpha_1\|A_*\|}{\delta}<1
\]

This implies that every eigenvalue $\lambda$ of $X$ satisfies $|\lambda|<1$. Hence no real eigenvalue can cross $-1$, and complex eigenvalues occur in conjugate pairs. Therefore
\[
\det(I_d+X)=\prod_{\lambda\in\sigma(X)}(1+\lambda)>0.
\]
It follows that
\[
f(\alpha_1)=\det D \cdot \det(I_d+X)<0.
\]

\medskip
\noindent
\textit{Step 4: Conclusion.}
Since $f(0)=1>0$ and $f(\alpha_1)<0$, there exists $\alpha_*\in(0,\alpha_1)$ such that
\[
f(\alpha_*)=0.
\]
Finally, by the strict monotonicity of $\alpha(t)$, there exists $t_*>0$ such that $\alpha(t_*)=\alpha_*$, and hence
\[
\det\nabla\eta(t_*,x_*)=0.
\]
This completes the proof.
\end{proof}

%
%
%
%
%

\section{General communication kernels}\label{sec:non_const}

In this section, we consider the Euler--alignment dynamics \eqref{EA} with a non-constant communication kernel $\phi$. In contrast to the constant-kernel case, one no longer has an explicit representation of the flow Jacobian $J(t,x)=\det\nabla\eta(t,x)$ purely in terms of $\nabla u_0(x)$. Thus, instead of a sharp pointwise spectral characterization, our goal is to derive quantitative sufficient criteria for the degeneration of the Lagrangian flow.  As in the constant-kernel analysis, we first treat the two-dimensional case, where the algebraic structure remains more explicit, and then turn to the general case $d\ge3$.

%
%
%
%
%

\subsection{Algebraic decomposition of the Jacobian and perturbative bounds}

We begin by isolating the leading damping mechanism in the evolution of $\nabla v$ and collecting the genuinely nonlocal contributions into a perturbative remainder. Differentiating the Lagrangian velocity equation in \eqref{LA}, we obtain
\[
\pa_t \nabla v(t,x) = - \kappa (\phi * \rho)(t,\eta(t,x)) \nabla v(t,x) + G(t,x),
\]
where
\[
G(t,x) := \kappa \intr \nabla\eta(t,x) \lt(\nabla\phi(\eta(t,x)-\eta(t,y))\otimes (v(t,y)-v(t,x))\rt)\rho_0(\dy).
\]
 Along each Lagrangian trajectory
$x\mapsto \eta(t,x)$, this is a linear equation with a time-dependent scalar
damping coefficient
\[
a(t,x):=\kappa(\phi * \rho)(t,\eta(t,x))\ge0.
\]

Introducing the integrating factor
\[
\Gamma(t,s;x):=\exp\lt(-\int_s^t a(\tau,x)\,\rd\tau\rt),
\]
we may write $\nabla v$ as 
\[
\nabla v(t,x) = \Gamma(t,0;x)\,\nabla u_0(x) + \int_0^t \Gamma(t,s;x)\,G(s,x)\,\ds.
\]
Integrating in time and using $\pa_t\nabla\eta=\nabla v$, we obtain
\[
\nabla\eta(t,x) = I_d + \alpha(t,x)\,\nabla u_0(x) + R(t,x),
\]
where the \emph{effective time} $\alpha(t,x)$ and the remainder $R$ are expressed 
in terms of $\displaystyle \beta(t,s;x) := \int_s^t \Gamma(\tau,s;x)\,\rd\tau$,
\[
\alpha(t,x):=\beta(t,0; x), \quad  R(t,x):=\int_0^t \beta(t,s;x)\,G(s,x)\,\ds.
\]
Since $\rho(t)\in\calP(\R^d)$ and $\phi\in L^\infty(\R^d)$, we have the uniform bound
\[
0 \leq (\phi * \rho)(t,z)\le \|\phi\|_{L^\infty}
\quad\text{for all }(t,z)\in[0,\infty)\times\R^d,
\]
and thus
\[
e^{-\kappa\|\phi\|_{L^\infty}(t-s)} \le \Gamma(t,s;x) \le 1.
\]
Consequently, the effective time admits the bounds
\[
\alpha_\infty(t) := \frac{1-e^{-\kappa\|\phi\|_{L^\infty}t}}{\kappa\|\phi\|_{L^\infty}} =\int_0^t e^{-\kappa\|\phi\|_{L^\infty}s}\,\ds \le \alpha(t,x)\le t,
\]
which is uniform in $x$.

Then, we decompose
\bq\label{eq:decom_gen}
\nabla\eta(t,x) = I_d+\alpha(t,x) \nabS u_0(x) + E(t,x), \quad  E(t,x):=
\alpha(t,x) \nabA u_0(x)+ R(t,x),
\eq
and this gives the following determinant representation
\[
J(t,x)=\det\nabla\eta(t,x)= \det\lt(I_d+\alpha(t,x)\,\nabS u_0(x)+  E(t,x)\rt),
\]
which will serve as the basic decomposition for the supercritical analysis in the subsequent subsections.

\begin{lemma}\label{lem:gro} Consider the system \eqref{LA}. Assume $\rho_0\in C^0\cap\calP(\R^d)$, $\phi \in W^{1,\infty}(\R^d; \R_{\ge 0})$ and $u_0\in C^1\cap W^{1,\infty}(\R^d;\R^d)$. Fix $x_* \in \supp (\rho_0)$. Then, for all $t \geq 0$ we have
\[
\|R(t,x_*)\| \leq \frac{C_1^*}{C_0^2} \lt( e^{C_0 t} - 1 - C_0 t\rt),
\]
where $C_0, C_1^*$ are positive constants given by
\[
C_0 := \max\{1, \kappa\|\nabla \phi\|_{L^\infty} \rdv \}, \quad C_1^* := \kappa\|\nabla\phi\|_{L^\infty} \rdv (1 + \|\nabla u_0(x_*)\|).
\]
\end{lemma}
\begin{proof}
 Introduce the diameter of position and velocity quantities associated to \eqref{LA} as
 \[
 \rd_\eta(t): = \esssup_{x,y \in \supp(\rho_0)}|\eta(t,x) - \eta(t,y)|, \quad  \rd_v(t): = \esssup_{x,y \in \supp(\rho_0)}|v(t,x) - v(t,y)|.
 \]
A standard maximum principle yields
 \[
 \rd_v(t) \leq \rdv \quad \text{for all } t \geq 0.
 \]
 
Differentiate the system \eqref{LA} with respect to $x$, we find
\begin{align*}
\pa_t \nabla \eta &= \nabla v, \cr
\pa_t \nabla v &= \kappa \intr \nabla \eta(x)  \lt(\nabla\phi(\eta(x) - \eta(y))\otimes(v(y) - v(x))\rt) \rho_0(\dy) \cr
&\quad  - \kappa \intr \phi(\eta(x) - \eta(y))   \nabla v(x)\,\rho_0(\dy).
\end{align*}
We first readily find
\[
\frac{\rd}{\dt} \|\nabla \eta(t,x_*)\| \leq \|\nabla v(t,x_*)\|,
\]
For $\|\nabla v\|$, we note that 
\[
\lt|\intr \nabla \eta(x_*)  \lt(\nabla\phi(\eta(x_*) - \eta(y))\otimes(v(y) - v(x_*))\rt) \rho_0(\dy)\rt| \leq \|\nabla \eta(x_*)\|\|\nabla\phi\|_{L^\infty} \rd_v(t),
\] 
and since $\phi \geq 0$, we obtain
\[
\frac{\rd}{\dt}\|\nabla v(t,x_*) \|  \leq \kappa \|\nabla\phi\|_{L^\infty}\|\nabla \eta(t,x_*) \|  \rdv  
\] 
Thus, we arrive at
\[
\frac{\rd}{\dt} \lt( \|\nabla \eta(t,x_*)\| + \|\nabla v(t,x_*)\| \rt) \leq C_0 \lt( \|\nabla \eta(t,x_*)\| + \|\nabla v(t,x_*)\|\rt),
\]
where $C_0 > 0$ is given by
\[
C_0 = \max\{1, \kappa\|\nabla \phi\|_{L^\infty} \rdv \}.
\]
An application of Gr\"onwall's lemma yields
\[
\|\nabla \eta(t,x_*) \| + \|\nabla v(t,x_*)\| \leq \lt( 1 + \|\nabla u_0(x_*)\| \rt ) e^{C_0 t},
\]

Using the above estimate together with the definition of $R$, we obtain
\[
\|R(t,x_*)\| \le \int_0^t (t-s)\|G(s,x_*)\|\,\ds \le C_1^* \int_0^t (t-s) e^{C_0 s}\,\ds = \frac{C_1^*}{C_0^2} \lt(e^{C_0 t} - 1 - C_0 t\rt),
\]
where we used
\begin{align*}
\|G(s,x_*)\| &\leq \kappa \|\nabla \eta(s,x_*)\| \|\nabla\phi\|_{L^\infty} \rdv \cr
&\leq \kappa\|\nabla\phi\|_{L^\infty} \rdv\lt(1 + \|\nabla u_0(x_*)\|\rt) e^{C_0 s} \cr
&= C_1^* e^{C_0 s}.
\end{align*}
This completes the proof.
\end{proof}

\begin{remark} 
A key feature of the general-kernel case is that the remainder term $R(t,x)$ in \eqref{eq:decom_gen} is not controlled uniformly in time by the present argument. In particular, Lemma \ref{lem:gro} yields a bound that grows with $t$, so the method does not provide a global-in-time critical-threshold theory in the classical sense.

This is not the objective here. The criterion is instead a finite-time perturbative criterion. At the reference time $t_1$ chosen in Theorem \ref{thm:gen_phi_d}, the leading symmetric deformation
\[
I_d+\alpha(t_1,x_*)\nabS u_0(x_*)
\]
has crossed the algebraic degeneracy threshold with a prescribed spectral margin. The assumptions then require the skew-symmetric contribution and the nonlocal remainder to remain smaller than this margin. Under this balance, the sign change of the leading determinant persists under the perturbation, and the intermediate value argument yields a time $t_*\in(0,t_1)$ at which
\[
\det\nabla\eta(t_*,x_*)=0.
\]
Thus the result should be understood as a quantitative sufficient condition for finite-time Lagrangian degeneracy, rather than as a sharp global critical-threshold characterization.
\end{remark}

%
%
%
%
%

\subsection{Two-dimensional supercritical initial data}

We specialize to $d=2$ and establish a   perturbative finite-time criterion ensuring that the Jacobian $J(t,x)=\det\nabla\eta(t,x)$ becomes negative at a controlled time   along a characteristic issued from a point $x_*\in\supp\rho_0$. The argument relies on the decomposition of $\nabla\eta(t,x_*)$ into a leading symmetric part determined by $\nabS u_0(x_*)$ and a perturbative term collecting the rotational component together with the nonlinear effects generated by the spatial variability of $\phi$. Exploiting the explicit determinant expansion available for $2\times2$ matrices, we identify a time $t_1>0$ at which the leading part already has negative determinant, while the perturbation remains sufficiently small. This yields a quantitative condition under which $\det\nabla\eta(t_1,x_*)<0$, and hence $\det\nabla\eta(t_*,x_*)=0$ for some $t_*\in(0,t_1)$.

\begin{theorem}\label{thm:2d_gen_phi}
Let $d=2$ and consider the Lagrange--alignment system \eqref{LA} with $\rho_0\in C^0\cap\calP(\R^2)$, $\phi\in W^{1,\infty}(\R^2;\R_{\ge0})$, and $u_0\in C^1\cap W^{1,\infty}(\R^2;\R^2)$. For $x\in\supp\rho_0$, let $\mu_1(x)\le\mu_2(x)$ be the eigenvalues of $\nabS u_0(x)$, and set $\omega_0(x):=(\nabla\times u_0)(x)$. Assume that there exist $x_*\in\supp\rho_0$, $\delta>0$, and $t_1>0$ such that, 
\[
L_*:=-\mu_1(x_*) >\kappa(1+\delta)\|\phi\|_{L^\infty},
\qquad
\alpha_\infty(t_1) =\frac{1-e^{-\kappa\|\phi\|_{L^\infty}t_1}} {\kappa\|\phi\|_{L^\infty}} =\frac{1+\delta}{L_*},
\]
and in addition that, $1+\mu_2(x_*)t_1>0.$\footnote{Which is required only when $\mu_2(x_*)<0$.}\newline
Suppose further the following smallness hypothesis
\bq\label{eq:small_gen2d}
 E_{t_1}(x_*) < \sqrt{N_*^2+ \delta m_{2,*}} - N_*, \  E_{t_1}(x_*) := t_1\frac{|\omega_0(x_*)|}{2} +\frac{C_1^*}{C_0^2} \lt(e^{C_0t_1}-1-C_0t_1\rt),
\eq
where  $C_0, C_1^*>0$ are the constants specified  in Lemma \ref{lem:gro}, and $m_{2,*}, N_*$ are given by
\[
m_{2,*}:= \begin{cases}
1, & \mu_2(x_*)\ge0,\\
1+\mu_2(x_*)t_1, & \mu_2(x_*)<0,
\end{cases}
\qquad
N_*:=\max\lt\{L_*t_1-1,\ 1+|\mu_2(x_*)|t_1\rt\}.
\]
 Then there exists $t_*=t_*(x_*)\in(0,t_1)$ such that
\[
\det\nabla\eta(t_*,x_*)=0.
\]
\end{theorem}

\begin{proof}
 Let $x_*\in\supp\rho_0$ be a point satisfying the assumptions. By the algebraic decomposition and effective-time representation \eqref{eq:decom_gen}, we may write
\[
\nabla\eta(t,x_*) = I_2 + \alpha(t,x_*) S_* +  E_*(t), \quad   E_*(t):= \alpha(t,x_*) A_*+  R(t,x_*),
\]
where $A_* = \nabA u_0(x_*)$, and the effective time $\alpha(t,x)$ satisfies the uniform bounds
\[
\alpha_\infty(t)\le \alpha(t,x)\le t.
\]
Moreover, by Lemma \ref{lem:gro}, the remainder $R$ obeys the estimate
\[
\|R(t,x_*)\| \le \frac{C_1^*}{C_0^2}\lt(e^{C_0 t}-1-C_0 t\rt).
\]

We now evaluate the decomposition at the time $t_1$ chosen so that the lower bound $\alpha_\infty(t_1)=(1+\delta)/L_*$ with $L_*=-\mu_1>0$ forces the most compressive factor $1+\mu_1\alpha(t_1,x_*)$ to be negative with margin at least $-\delta$. Choosing an orthonormal eigenbasis of $S_*$ so that $S_*=\diag(\mu_1,\mu_2)$, we introduce
\[
D_*:=I_2+\alpha(t_1,x_*)\,S_* =\diag\lt(1+\mu_1\alpha(t_1,x_*),\,1+\mu_2\alpha(t_1,x_*)\rt).
\]
Since $\alpha(t_1,x_*)\ge\alpha_\infty(t_1)=(1+\delta)/L_*$, we have $1+\mu_1\alpha(t_1,x_*)=1-L_*\alpha(t_1,x_*)\le-\delta<0$. If $\mu_2\ge0$, then trivially $1+\mu_2\alpha(t_1,x_*)\ge1$. If $\mu_2<0$, using $\alpha(t_1,x_*)\le t_1$ and the additional assumption $1+\mu_2 t_1>0$, we obtain $1+\mu_2\alpha(t_1,x_*)\ge 1+\mu_2 t_1>0$. In both cases, it follows that
\bq\label{eq:detD_upper}
\det D_* =\lt(1+\mu_1\alpha(t_1,x_*)\rt)\lt(1+\mu_2\alpha(t_1,x_*)\rt) \le (-\delta)\,m_2<0.
\eq
where $m_2=1$ if $\mu_2\ge0$ and $m_2=1+\mu_2 t_1$ if $\mu_2<0$.

Next, since $\alpha(t_1,x_*)\le t_1$, we estimate
\[
|1+\mu_1\alpha(t_1,x_*)|\le L_*t_1-1, \quad |1+\mu_2\alpha(t_1,x_*)|\le 1+|\mu_2|t_1,
\]
and hence
\bq\label{eq:D_norm_bound}
\|D_*\|\le N=\max\{L_*t_1-1,\ 1+|\mu_2|t_1\}.
\eq
In two dimensions,
\[
A_*=\frac12
\begin{pmatrix}
0&-\omega_*\\
\omega_*&0
\end{pmatrix},
\quad
\|A_*\|=\frac{|\omega_*|}{2},
\]
and using $\alpha(t_1,x_*)\le t_1$ together with the perturbative bound on $R$, we find
\bq\label{eq:E_bound}
\begin{split}
\|E_*(t_1)\|
&\le \alpha(t_1,x_*)\|A_*\|+\|R(t_1,x_*)\| \\
&\le t_1\frac{|\omega_*|}{2}+\frac{C_1^*}{C_0^2}\lt(e^{C_0t_1}-1-C_0t_1\rt)
=E_{t_1}.
\end{split}
\eq

Since $\nabla\eta(t_1,x_*)=D_*+E_*(t_1)$, we use the exact determinant identity for $2\times2$ matrices,
\[
\det(D_*+E_*)=\det D_*+\tr(\cof(D_*)^\top E_*)+\det E_*,
\]
together with the elementary bounds $|\tr Z|\le2\|Z\|$, $\|\cof(D_*)\|=\|D_*\|$, and $|\det E_*|\le\|E_*\|^2$, to obtain
\[
\det\nabla\eta(t_1,x_*) \le \det D_*+2\|D_*\| \|E_*(t_1)\|+\|E_*(t_1)\|^2.
\]
Applying \eqref{eq:detD_upper}, \eqref{eq:D_norm_bound}, and \eqref{eq:E_bound} yields
\[
\det\nabla\eta(t_1,x_*) \le -\delta m_2 + 2N E_{t_1}+E_{t_1}^2.
\]
The smallness hypothesis \eqref{eq:small_gen2d} therefore implies $\det\nabla\eta(t_1,x_*)<0$.

Finally, the map $t\mapsto\det\nabla\eta(t,x_*)$ is continuous and satisfies $\det\nabla\eta(0,x_*)=\det I_2=1>0$, hence by the intermediate value theorem there exists $t_*\in(0,t_1)$ such that 
\[
\det\nabla\eta(t_*,x_*)=0.
\]
This completes the proof.
\end{proof}

\begin{remark}[A concrete sufficient scenario in two dimensions]
We record a simple consequence of Theorem \ref{thm:2d_gen_phi} in the case where the second eigenvalue of the symmetric part is nonnegative. Assume that
\[
\mu_2\ge0,\quad L_*:=-\mu_1>0,\quad a:=\kappa\|\phi\|_{L^\infty},\quad \omega_*:=(\nabla\times u_0)(x_*).
\]
Fix $\delta=1$ and let $t_1>0$ be determined by
\[
\alpha_\infty(t_1)=\frac{1-e^{-a t_1}}{a}=\frac{2}{L_*}.
\]
Assume in addition that
\[
L_*\ge4a\quad\text{and}\quad L_*\ge4C_0.
\]
Then
\[
t_1=\frac1a\log\frac1{1-\frac{2a}{L_*}}\le\frac{4}{L_*}.
\]
In particular, $C_0t_1\le1$, and hence $e^z-1-z\le(e/2)z^2$ for $0\le z\le1$ gives
\[
E_{t_1}\le \frac{|\omega_*|}{2}t_1+\frac{eC_1^*}{2}t_1^2
\le \frac{2|\omega_*|}{L_*}+\frac{8eC_1^*}{L_*^2}
=:B(L_*).
\]
Moreover, since $\mu_2\ge0$, we have $m_2=1$, and
\[
N=\max\{L_*t_1-1,\ 1+\mu_2t_1\}\le 4\lt(1+\frac{\mu_2}{L_*}\rt).
\]
Thus the smallness condition \eqref{eq:small_gen2d} is ensured by
\bq\label{eq:BL}
8\lt(1+\frac{\mu_2}{L_*}\rt)B(L_*)+B(L_*)^2<1.
\eq
Consequently, if $\mu_2/L_*$ and $|\omega_*|/L_*$ remain controlled, the assumptions of Theorem \ref{thm:2d_gen_phi} are satisfied for sufficiently large $L_*=-\mu_1$. In particular, a sufficiently strong negative eigenvalue of $\nabS u_0(x_*)$ forces finite-time Lagrangian degeneracy even when no negative divergence condition is imposed. This condition is still non-sharp, but it is more explicit than the dimension-uniform criterion because the two-dimensional determinant identity gives the sharper perturbative condition \eqref{eq:small_gen2d}.

Finally, we note that the dependence of $C_1^*$ on $\|\nabla u_0(x_*)\|$ does not create a circular obstruction. Indeed, in the above estimate this dependence enters only through the term $\frac{8eC_1^*}{L_*^2}$ in $B(L_*)$. In particular, when $\mu_2/L_*$ remains controlled, the condition \eqref{eq:BL} is ensured by taking $L$ larger than a constant multiple of $|\omega_*|+\sqrt{C_1^*}$. Thus the contribution of $C_1^*$ is only of square-root order. Since $C_1^*$ depends linearly on $\|\nabla u_0(x_*)\|$, a sufficiently strong compressive eigenvalue $L_*=-\mu_1$ still enforces finite-time degeneracy; the dependence on the full local gradient is lower order than the leading compression.
\end{remark}

\begin{remark}[Comparison with 2D critical-threshold blow-up results]\label{rem:2d_CT_blowup_comp}
In \cite{TT14},  finite-time blow-up for the two-dimensional Euler--alignment system, i.e., \eqref{EA} with $d=2$, is obtained under a \emph{negative divergence condition} on the initial data, namely
\[
\inf_{x \in \supp(\rho_0)} (\nabla \cdot u_0)(x) < -\frac12 \lt(\kappa + \sqrt{\kappa^2 + 4 \kappa\, \rdv  \|\phi\|_{\dot W^{1,\infty}}} \rt)
\]
together with additional structural assumptions on the off-diagonal components of
$\nabla u_0$. Under these hypotheses, it is shown that $\inf_{\supp\rho(t,\cdot)}(\nabla \cdot u)(t,\cdot)\to-\infty$ in finite time.

In contrast, Theorem \ref{thm:2d_gen_phi} does \emph{not} impose any sign condition on the initial divergence. Instead, our criterion is formulated in terms of the local spectral data of the symmetric gradient $S_*=\nabS u_0(x_*)$, requiring only that its minimal eigenvalue $\mu_1$ be sufficiently negative. In particular, the divergence $\mu_1+\mu_2$ at $x_*$ may be positive.

A key distinction lies in the role of the off-diagonal components of $\nabla u_0$. In our approach, the skew-symmetric part $A_*=\nabA u_0(x_*)$ enters only through its scalar size $|\omega_*|$ in the perturbative quantity $E_{t_1}$. No sign condition, lower bound, or structural coupling involving the off-diagonal entries is imposed. Large rotational components are therefore not excluded a priori; they simply make the perturbative smallness requirement more stringent. This is in clear contrast with divergence-based critical-threshold theories, where the evolution of divergence is coupled more directly to the vorticity through invariant-region arguments.
\end{remark}

%
%
%
%
%

\subsection{Higher-dimensional supercritical initial data}

We now extend the   perturbative degeneracy analysis to arbitrary dimension $d\ge2$. In contrast to the two-dimensional case, no exact low-dimensional determinant expansion is available, so we adopt a dimension-uniform perturbative strategy. The basic idea is to decompose $\nabla\eta(t,x_*)$ into a diagonal leading part generated by the symmetric gradient $\nabS u_0(x_*)$ and a perturbation collecting both the skew-symmetric component and the nonlinear remainder induced by the spatial variability of $\phi$. The sign of the determinant is then inferred by factoring out the leading part and controlling the residual factor through an operator-norm estimate. Although this argument is less sharp than the two-dimensional one, it yields a clean sufficient criterion valid in all dimensions.  

 We now prove the dimension-uniform criterion stated in Theorem \ref{thm:gen_phi_d}.

\begin{proof}[Proof of Theorem \ref{thm:gen_phi_d}]
We evaluate the decomposition \eqref{eq:decom_gen} at the prescribed time $t_1$ and show that the leading diagonal part has a negative determinant, while the perturbation remains too small to change its sign.   Let $x_*\in\supp\rho_0$ be a point satisfying the assumptions. By the decomposition \eqref{eq:decom_gen}, for every $t\ge0$ we may write
\[
\nabla\eta(t,x_*) = I_d+\alpha(t,x_*)\,S_*+E_*(t),
\quad
E_*(t):=\alpha(t,x_*)\,A_*+R(t,x_*).
\]
The effective time satisfies the uniform bounds
\[
\alpha_\infty(t)\le \alpha(t,x_*)\le t,
\]
and applying Lemma \ref{lem:gro} at time $t=t_1$ yields the estimate
\bq\label{eq:E_gend_comp_alpha}
\|E_*(t_1)\|
\le
\alpha(t_1,x_*)\|A_*\|+\|R(t_1,x_*)\|
\le
t_1\|A_*\|+\frac{C_1^*}{C_0^2}\lt(e^{C_0t_1}-1-C_0t_1\rt).
\eq

Let $Q\in O(d)$ be an orthogonal matrix diagonalizing the symmetric part $S_*$, so that $Q^\top S_*Q=\diag(\mu_1,\dots,\mu_d)$. Setting $\alpha_*:=\alpha(t_1,x_*)\in[\alpha_\infty(t_1),t_1]$ and defining
\[
D:=I_d+\alpha_*\,Q^\top S_*Q = \diag(1+\alpha_*\mu_1,\dots,1+\alpha_*\mu_d),
\quad
\widetilde E:=Q^\top E_*(t_1)Q,
\]
the invariance of the determinant and operator norm under orthogonal conjugation gives
\[
\det\nabla\eta(t_1,x_*)=\det(D+\widetilde E),
\quad
\|\widetilde E\|=\|E_*(t_1)\|.
\]

By the separation condition \eqref{eq:sep_gen_d_alpha}, the matrix $D$ is invertible and
\[
\|D^{-1}\| = \max_{1\le j\le d}\frac{1}{|1+\alpha_*\mu_j|} \le \frac{1}{\delta}.
\]
Moreover, \eqref{eq:sep_gen_d_alpha} prevents any sign change of the factors $1+\alpha\mu_j$ for $\alpha\in[\alpha_\infty(t_1),t_1]$, so the parity of the number of negative diagonal entries of $D$ is constant on this interval. Since
\[
r = \# \{ j : 1+\alpha_\infty(t_1)\mu_j<0\}
\]
is odd by assumption, it follows that
\[
\det D=\prod_{j=1}^d(1+\alpha_*\mu_j)<0,
\quad
|\det D|=\prod_{j=1}^d|1+\alpha_*\mu_j|\ge \delta^d,
\]
and in particular $\det D\le -\delta^d<0$.

Factoring out the leading diagonal part, we write
\[
\det(D+\widetilde E)=\det D\cdot \det(I_d+X), \quad X:=D^{-1}\widetilde E.
\]
Combining \eqref{eq:E_gend_comp_alpha} with the bound on $\|D^{-1}\|$, we obtain
\[
\|X\| \le \frac{\|E_*(t_1)\|}{\delta} \le \frac{t_1\|A_*\|+\frac{C_1^*}{C_0^2}(e^{C_0t_1}-1-C_0t_1)}{\delta}  \le \frac{t_1\|A_*\| + C_1^* t_1^2 e^{C_0 t_1}}{\delta}.
\]
The smallness condition \eqref{eq:small_d_alpha} therefore ensures that $\|X\|<1$.

Since $\|X\|<1$, the same continuity argument as in the constant-kernel case implies
$\det(I_d+X)>0$. Consequently,
\[
\det\nabla\eta(t_1,x_*) = \det(D+\widetilde E) = \det D\cdot \det(I_d+X) < 0.
\]
As $\det\nabla\eta(0,x_*)=\det I_d=1>0$ and the map $t\mapsto\det\nabla\eta(t,x_*)$ is continuous, there exists $t_*\in(0,t_1)$ such that
\[
\det\nabla\eta(t_*,x_*)=0.
\]
The density blow-up assertion follows from Proposition \ref{prop:one_sided_density} below.  This completes the proof.
\end{proof}

\begin{remark}[Relation with the two-dimensional case]
Although Theorem \ref{thm:gen_phi_d} applies for all $d\ge2$, its specialization to $d=2$ is strictly more conservative than Theorem \ref{thm:2d_gen_phi}. Indeed, the present argument relies only on the factorization
\[
\det(D+E)=\det D\,\det(I_2+D^{-1}E),
\]
together with a uniform operator-norm smallness condition ensuring that the perturbation does not alter the sign of the determinant. This yields a dimension-independent criterion, but one that is necessarily restrictive.

By contrast, in two dimensions one can use the exact identity
\[
\det(D+E)=\det D+\tr(\cof(D)^\top E)+\det E,
\]
which captures the perturbation more precisely and leads to the sharper condition \eqref{eq:small_gen2d}. This refinement is specific to the $2\times2$ structure and has no direct analogue in higher dimensions.

In particular, when $d=2$, the parity condition in Theorem \ref{thm:gen_phi_d} reduces to the requirement that exactly one of the factors $1+\alpha_\infty(t_1)\mu_j$ is negative, in agreement with the two-dimensional framework. The additional separation condition \eqref{eq:sep_gen_d_alpha} imposed in Theorem \ref{thm:gen_phi_d} is stronger than the corresponding assumption in Theorem \ref{thm:2d_gen_phi}, and is introduced to ensure a uniform perturbative control in arbitrary dimension.
\end{remark}

We conclude this section with a general lower blow-up estimate. Independently of the specific supercritical criterion used above, if the Lagrangian Jacobian vanishes for the first time at some finite $t_*>0$ along a classical trajectory, then the density necessarily blows up at least at the rate $(t_*-t)^{-1}$ along that characteristic.
 \begin{proposition}\label{prop:one_sided_density}  
 Consider the Lagrange--alignment system \eqref{LA}. Assume $\rho_0\in C^0\cap\calP(\R^d)$ and $u_0\in C^1\cap W^{1,\infty}(\R^d;\R^d)$ and fix $x_*\in\supp(\rho_0)$. Suppose that the first breakdown time at $x_*$ is $t_*\in(0,\infty)$, i.e.
\[
J(t,x_*)>0\ \text{for }0\le t<t_*, \quad J(t_*,x_*)=0.
\]
Then there exists a constant $C>0$ such that
\[
0<J(t,x_*)\le C(t_*-t)\quad\text{for all $t < t_*$.}
\]
In particular, if $\rho_0(x_*)>0$, then for $t<t_*$,
\[
\rho(t,\eta(t,x_*)) \ge \frac{c}{t_*-t}
\]
for some $c>0$.
\end{proposition}
\begin{proof}
Fix $x_*\in\supp\rho_0$. Note that $\nabla\eta(t,x_*)$ is $C^1$ in time and satisfies
\[
\partial_t \nabla\eta(t,x_*)=\nabla v(t,x_*).
\]
Let $m_1(t),\dots,m_d(t)\in\R^d$ denote the column vectors of $\nabla\eta(t,x_*)$, so that
$J(t)=\det(m_1(t),\dots,m_d(t))$.
By multilinearity of the determinant in the columns, we may differentiate:
\[
J'(t) = \det(\nabla\eta(t,x_*))' = \sum_{k=1}^d \det\big(m_1(t),\dots,m_{k-1}(t),m_k'(t),m_{k+1}(t),\dots,m_d(t)\big).
\]
Using Hadamard's inequality $|\det(a_1,\dots,a_d)|\le \prod_{j=1}^d |a_j|$,
we obtain
\[
|J'(t)|
\le \sum_{k=1}^d \bigg(\prod_{j\neq k}|m_j(t)|\bigg)\,|m_k'(t)|.
\]
Now note that $|m_j(t)|\le \|\nabla\eta(t,x_*)\|$ for all $j$, and $|m_k'(t)|\le \|\nabla v(t,x_*)\|$.
Hence, we have
\[
|J'(t)| \le  d\,\|\nabla\eta(t,x_*)\|^{d-1}\,\|\nabla v(t,x_*)\|.
\]
By Lemma \ref{lem:gro}, there exists a constant $C>0$ such that
\[
|J'(t)|\le C \quad\text{for all }0\le t<t_*.
\]

Since $J(t_*)=0$, the fundamental theorem of calculus gives, for any $t<t_*$,
\[
|J(t)|=|J(t)-J(t_*)| = \lt|\int_t^{t_*} J'(s)\,\ds\rt| \le \int_t^{t_*} |J'(s)|\,\ds \le C(t_*-t).
\]
This, together with $J(t)>0$ for $t<t_*$ by hypothesis, yields
\[
0<J(t,x_*)\le C(t_*-t)
\]
for all $t<t_*$.

Finally, if $\rho_0(x_*)>0$, then combining this with $J(t,x_*)\le C(t_*-t)$ yields
\[
\rho(t,\eta(t,x_*)) =\frac{\rho_0(x_*)}{J(t,x_*)} \ge \frac{\rho_0(x_*)}{C}\,\frac1{t_*-t},
\]
which is the claimed one-sided blow-up lower bound with $c=\rho_0(x_*)/C$.
\end{proof}


%
%
%
%

\bibliographystyle{abbrv}
\bibliography{CT_EA}

\end{document}